\RequirePackage{fix-cm}
\documentclass[10pt]{amsart}
\usepackage[utf8]{inputenc}
\usepackage[T1]{fontenc}

\usepackage{soul}

\usepackage{graphicx}
\usepackage{longtable}
\usepackage{float}
\usepackage{wrapfig}
\usepackage{textcomp}
\usepackage{marvosym}
\usepackage{latexsym}
\usepackage{amssymb}
\usepackage{xargs}
\usepackage[sort&compress,numbers]{natbib}
\usepackage{algorithm2e}
\usepackage{xcolor}
\usepackage{subcaption}
\definecolor{HalfGray}{gray}{0.55}
\definecolor{OliveGreen}{rgb}{0,.35,0}
\definecolor{webbrown}{rgb}{.6,0,0}
\definecolor{BrightViolet}{rgb}{0.5,0.2,0.8}
\definecolor{Maroon}{cmyk}{0, 0.87, 0.68, 0.32}
\definecolor{RoyalBlue}{cmyk}{1, 0.50, 0, 0.25}
\definecolor{Black}{cmyk}{0, 0, 0, 0}

\definecolor{ccccccc}{RGB}{204,204,204}
\definecolor{c808080}{RGB}{128,128,128}
\definecolor{c999999}{RGB}{153,153,153}
\definecolor{ce6e6e6}{RGB}{230,230,230}

\usepackage[unicode]{hyperref}
\hypersetup{
colorlinks=true,
linktocpage=true,
pdfstartview=FitH,
breaklinks=true,
pdfpagemode=UseNone,
pageanchor=true,
pdfpagemode=UseOutlines,
plainpages=false,
bookmarksnumbered,
bookmarksopen=false,
bookmarksopenlevel=1,
hypertexnames=true,
pdfhighlight=/O,
urlcolor=Maroon, linkcolor=RoyalBlue, citecolor=OliveGreen, 
pdftitle={},
pdfauthor={},
pdfsubject={},
pdfkeywords={},
pdfcreator={pdfLaTeX},}
\usepackage{pgfplots}
\usepackage{amsthm,amsmath,enumerate,bbm,tikz,tabularx,multicol,environ}
\usetikzlibrary{matrix,arrows,calc,patterns}
\usepgfplotslibrary{groupplots,dateplot}
\usetikzlibrary{patterns,shapes.arrows}
\pgfplotsset{compat=newest}

\makeatletter
\newsavebox{\measure@tikzpicture}
\NewEnviron{scaletikzpicturetowidth}[1]{%
  \def\tikz@width{#1}%
  \begin{lrbox}{\measure@tikzpicture}%
  \BODY
  \end{lrbox}%
  \pgfmathparse{#1/\wd\measure@tikzpicture}%
  \BODY
}
\makeatother

\newtheorem{theorem}{Theorem}[section]
\newtheorem{proposition}[theorem]{Proposition}
\newtheorem{corollary}[theorem]{Corollary}
\newtheorem{lemma}[theorem]{Lemma}
\theoremstyle{definition}
\newtheorem{definition}[theorem]{Definition}
\newcounter{assump}
\newtheorem{assumption}[assump]{Assumption}

\theoremstyle{remark}
\newtheorem{remark}[theorem]{Remark}

\DeclareMathOperator*{\argmin}{arg\,min}

\renewcommand{\leq}{\leqslant}
\renewcommand{\geq}{\geqslant}
\renewcommand{\le}{\leqslant}

\newcommand{\expec}[2]{\mathbb{E}_{#1}\left[ #2\right]}

\newcommand{\m}{\mathbb}

\newcommand{\E}{\mathbb{E}}
\newcommand{\si}{\sigma}
\newcommand{\ga}{\gamma}

\newcommand{\la}{\lambda}

\author{Joon Kwon, Yijun Wan \& Bruno Ziliotto}
\date{\today}
\title[Time-dependent approachability]{Time-dependent Blackwell approachability and application to absorbing games}
\begin{document}

\begin{abstract}
  Blackwell's approachability~\citep{blackwell1954controlled,blackwell1956analog} is a very general online learning framework where a Decision Maker obtains vector-valued outcomes, and aims at the convergence of the average outcome to a given ``target'' set. Blackwell gave a sufficient condition for the decision maker having a strategy guaranteeing such a convergence against an adversarial environment, as well as what we now call the Blackwell's algorithm, which then ensures convergence. Blackwell's approachability has since been applied to numerous problems, in regret minimization and game theory, in particular. We extend this framework by allowing the outcome function and the inner product to be time-dependent. We establish a general guarantee for the natural extension to this framework of Blackwell's algorithm. In the case where the target set is an orthant, we present a family of time-dependent inner products which yields different convergence speeds for each coordinate of the average outcome. We apply this framework to absorbing games (an important class of stochastic games) for which we construct $\varepsilon$-uniformly optimal strategies using Blackwell's algorithm in a well-chosen auxiliary approachability problem, thereby giving a novel illustration of the relevance of online learning tools for solving games.


\end{abstract}
\maketitle

\section{Introduction}
\label{sec:introduction}

The fundamental von Neumann's minimax theorem characterizes the
highest payoff that each player of a finite two-player zero-sum game
can guarantee. 
\citet{blackwell1954controlled,blackwell1956analog} proposed a
surprising extension of this result for repeated games with
vector-valued outcomes. For a given set called the target, he gave a
sufficient condition for the player having a strategy that guarantees
convergence in average of the outcomes to the target.
This condition is also necessary for convex sets. When the condition
is satisfied, the original algorithm proposed by Blackwell, as well as
many recent alternatives~\citep{hart2001general,gordon2007no,abernethy2011blackwell,kwon2021refined,shimkin2016online,grand2022solving,dann23pseudonorm}, do guarantee such a convergence.

This approachability framework has since found numerous applications in
online learning and game theory---see \citet{perchet2014approachability}
for a survey. The very first application was an alternative solution
to the fundamental sequential decision problem called regret
minimization
\citep{blackwell1954controlled,hannan1957approximation,cesa2006prediction}.
It has since been noticed that many variants and extensions of regret
minimization are also special cases, e.g.\ internal/swap
regret~\citep{stoltz2005internal,blum2005external}, with variable stage
duration~\citep{mannor2008regret}, with sleeping experts
\citep{blum2005external}, with fairness constraints
\citep{chzhen2021unified}, with global costs~\citep{even2009online},
etc. Further important areas of application include online resource
allocation problems such as scheduling~\citep{hou2009admission},
capacity pooling~\citep{zhong2018resource}, and reinforcement learning
\citep{mannor2004geometric,miryoosefi2019reinforcement,yu2021provably,li2021blackwell}.
A recent paper also proposed an interesting application of Blackwell's
algorithm to efficiently solving saddle-point
problems~\citep{grand2022solving}.

Blackwell's approachability has also been applied to game theory in
various ways. It was first observed that regret minimization can be
used to iteratively compute game solutions: \citet{freund1999adaptive}
constructively proved the von Neumann's minimax theorem using the
exponential weights algorithm. More generally, \citet{hart2000simple}
used Blackwell's approachability to define an algorithm
called \emph{Regret Matching} and showed convergence to the set of correlated
equilibria. Regret Matching turned out to be very effective in practice
and modern variants achieve state-of-the-art performance in solving large games
such as poker~\citep{zinkevich2007regret,tammelin2015solving,farina2020faster}.
Further applications include repeated games~\citep{mannor2009approachability}, with partial
monitoring~\citep{perchet2011approachability,mannor2011robust,mannor2013primal,mannor2014set}, with incomplete
information~\citep{kohlberg1975optimal} and coalitional
games~\citep{raja2020approachability}. Approachability was
also studied in the context of stochastic games with vector valued payoffs~\citep{FLP16,ragel22}.

Another important example of an (online) learning tool that has been successfully applied to games is reinforcement learning~\citep{samuel1959some}, which has been 
declined into multi-agent reinforcement learning in the context of stochastic games~\citep{littman1994markov}. Reinforcement learning methods typically need each state to be visited infinitely often. 
By contrast, we present a novel application of (an extension of) Blackwell's approachability to the important class of \emph{absorbing games}, for which such assumption cannot hold, because of irreversible state transitions. 

Zero-sum stochastic games~\citep{shapley1953stochastic} feature two
players who repeatedly play a zero-sum game with an outcome function
that depends on a state variable. The state follows a
Markov chain controlled by both players. In the $T$-stage game, the payoff function is the expected average payoff over $T$ stages, and we call $v_T$ its value.
A large part of the literature studies the asymptotic properties of this model
as the number of stages $T$ grows to infinity---see \citet{MSZ,sorin02b,SZ16,solan22} for general references. 
A subclass that has been intensively studied is \textit{absorbing games}, where the state can move at most once during the game, thus in an irreversible fashion. When the state space and the action sets are finite, \citet{kohlberg74} proved the
convergence of $v_T$ as $T\to + \infty$~\citep{BK76}---the limit $v_\infty$ is
called \textit{limit value}. Moreover, he proved the existence of the
\textit{uniform value}~\citep{MN81}, i.e.\  for each $\varepsilon>0$,
Player I (resp.\ Player II) has a strategy that guarantees $v_\infty-\varepsilon$
(resp.\ $v_\infty+\varepsilon$), for any sufficiently large $T$. Strategies satisfying such a property are called \textit{$\varepsilon$-uniform optimal}. The existence of the limit value and uniform value were generalized to any stochastic game with finite state space and finite action sets, respectively in \citet{BK76} and \citet{MN81}. The limit value may fail to exist when the state space is infinite~\citep{ziliotto16} or when one of the action sets is infinite~\citep{vigeral13}, even under strong topological assumptions. The situation is more positive in the case of absorbing games. Indeed, as far as the existence of limit value and uniform value is concerned, absorbing games with infinite state space can easily be reduced to a finite-state space setting. Moreover, absorbing games with compact action sets and separately continuous payoff and transition functions have a uniform value~\citep{MNR09}. The latter result was refined in \citet{HIR21}, who proved that there exist $\varepsilon$-uniform optimal strategies that can be generated by a finite-state space automaton with transition functions that only depend on the stage of the game. 

In this paper, we introduce an extension of
Blackwell's approachability framework where the outcome function and the inner product vary with time and apply it to obtain a new proof of the existence of the limit value and uniform value in absorbing games with compact action sets. As a result, we obtain a construction of $\varepsilon$-uniformly optimal  strategies that is fairly different from the ones in \citet{MNR09} and \citet{HIR21}. The reader can find a detailed comparison between the constructions in Subsection~\ref{sec:comp}. To our knowledge, this is the first time that approachability is used to design $\varepsilon$-uniformly optimal strategies. 
The connection between approachability and uniform value builds on the following simple idea. An $\varepsilon$-uniformly optimal strategy should induce good stage payoffs while ensuring that the state does not get absorbed into a bad state. This dual objective can be recast as an approachability problem with a two-dimensional vector payoff function. 
 The simplicity and generality of such an idea make it amenable to promising extensions, and we hope
that these tools and ideas will lead to a systematic approach for
solving other stochastic games and sequential decision problems.
In addition, our work constitutes another illustration of the relevance of online learning tools for solving games.
Moreover, possible applications of our framework to regret minimization are mentioned in Section~\ref{sec:conclusion}.

\subsection{Contributions and Summary}
\label{sec:contributions-summary}

We introduce in Section~\ref{sec:gener-appr-probl} an extension of
Blackwell's approachability framework where the outcome function and
the inner product vary with time. We define the associated Blackwell's
condition and Blackwell's algorithm.

In Section~\ref{sec:analysis}, we establish a general bound on the
distance of the average outcome to the target set,
measured by the time-dependent inner product. The bound holds as soon as
Blackwell's condition is satisfied and depends on the norm of each
past outcome, where outcome vector at time $t\geqslant 1$ is measured
with the corresponding norm at time $t$---this feature being essential in
the application to absorbing games.

In the case where the target set is an orthant, Corollary~\ref{cor:1}
proposes a family of time-dependent inner product which yields different
convergence speeds for each coordinate of the average outcome vector.

In Sections~\ref{sec:absorbing-games} to~\ref{sec:result-strat-play}, we recall the definition of absorbing
 games and showcase the above tools and results by applying them to the construction of $\varepsilon$-uniformly optimal strategies. 

\subsection{Related Work}
\label{sec:related-work}

Our presentation of Blackwell's approachability focuses on the case of
target sets which are closed convex cones, as in
\citet{abernethy2011blackwell} where the properties of the convex cone are
used to convert regret minimization algorithms
into approachability ones. Time-dependent outcome functions
also appear in \citet{lee2021online} with the restriction that they are
convex-concave.

\subsection{Notation}
\label{sec:notation} Let $d\geqslant 1$ be an integer.
For a vector $u\in \mathbb{R}^d$, we denote $u=(u^{(1)},\dots,u^{(d)})$ its components.

\section{A Generalized Approachability Framework}
\label{sec:gener-appr-probl}
We define an extension of Blackwell's approachability framework and
condition, and establish a guarantee for the corresponding Blackwell's
algorithm. Our approach allows the outcome function and the
inner product to vary with time.

Let $\mathcal{A},\mathcal{B}$ be action sets (with no particular
structure) for the Decision Maker and
Nature respectively. Let $d\geqslant 1$ be an integer and
$\mathcal{R}$ a subset of $(\mathbb{R}^d)^{\mathcal{A}\times \mathcal{B}}$.

The interaction goes as follows. At time $t\geqslant 1$,
\begin{itemize}
\item Nature chooses outcome function $\rho_t\in \mathcal{R}$ and reveals it to the Decision Maker,
\item the Decision Maker chooses action $a_t\in \mathcal{A}$,
\item Nature chooses action $b_t\in \mathcal{B}$,
\item outcome vector $r_t:=\rho_t(a_t,b_t)\in \mathbb{R}^d$ is revealed to the Decision Maker.
\end{itemize}

Formally, an algorithm of the Decision Maker is a sequence of maps
$(\sigma_t)_{t\geqslant 1}$ where for $t\geqslant 1$,
$\sigma_t:(\mathbb{R}^{d})^{t-1}\times \mathcal{R}^t\to \mathcal{A}$. For a given such
algorithm, a sequence of outcome functions $(\rho_t)_{t\geqslant 1}$ in $\mathcal{R}$ and a
sequence of Nature's actions $(b_t)_{t\geqslant 1}$ in
$\mathcal{B}$, the actions of the Decision Maker are then defined for
$t\geqslant 1$ as:
\begin{align*}
a_t&= \sigma_t(r_1,\dots,r_{t-1},\rho_1,\dots,\rho_t)\\
  &= \sigma_t\left( \rho_1(a_1,b_1),\dots,\rho_{t-1}(a_{t-1},b_{t-1}) ,\rho_1,\dots,\rho_t\right). 
\end{align*}

Our construction and analysis can be extended to general
closed convex target sets as discussed in Section~\ref{sec:extens-gener-clos}, we here restrict to closed convex cones for simplicity.
\begin{definition}
A nonempty subset $\mathcal{C}\subset \mathbb{R}^d$ is a \emph{closed convex
cone} if it is closed and if for all $z,z'\in \mathcal{C}$ and
$\lambda\in \mathbb{R}_+$, it holds that $z+z'\in \mathcal{C}$ and
$\lambda z\in \mathcal{C}$.
\end{definition}
\begin{remark}
It is immediate to verify that a closed convex cone is indeed convex.
Consequently, the orthogonal projection onto a closed convex
cone (with respect to any inner product) is well-defined.
\end{remark}

Let $\mathcal{C}\subset \mathbb{R}^d$ be a closed convex cone called
the \emph{target} and
$(\left< \,\cdot\,, \,\cdot\, \right>_{(t)} )_{t\geqslant 1}$ a
sequence of inner products in $\mathbb{R}^d$.
For $t\geqslant 1$, denote by $\left\|\,\cdot\,\right\|_{(t)}$ and $\pi_{(t)}^{\mathcal{C}}$
the associated Euclidean norm and the associated Euclidean projection
onto $\mathcal{C}$ respectively. In other words, for $r\in \mathbb{R}^d$,
\begin{align*}
\left\| r \right\|_{(t)}&=\left< r, r \right>_{(t)}^{1/2},\\
  \pi_{(t)}^{\mathcal{C}}(r)&=\argmin_{r'\in \mathcal{C}}\left\| r'-r \right\|_{(t)}.
\end{align*}

We now present the extension of Blackwell's condition~\citep{blackwell1954controlled,blackwell1956analog} to our framework.
\begin{definition}[Blackwell's condition]
$\mathcal{C}$ satisfies \emph{Blackwell's condition} (with respect to
$\mathcal{R}$ and
$(\left< \,\cdot\,, \,\cdot\, \right>_{(t)})_{t\geqslant 1}$) if for
all $t\geqslant 1$, $\rho\in \mathcal{R}$, $r\in \mathbb{R}^d$, there
exists $a\left[ t,\rho,r \right]\in \mathcal{A}$ such that for all $b\in \mathcal{B}$,
\[ \left< \rho\left(a\left[ t,\rho,r \right],b  \right), r-\pi_{(t)}^{\mathcal{C}}(r) \right>_{(t)}\leqslant 0. \]
The above map
\[ \begin{array}{rccc}
a:&\mathbb{N}^*\times \mathcal{R}\times \mathbb{R}^d &\longrightarrow  & \mathcal{A}\\
& (t,\rho,r)&\longmapsto  & a\left[ t,\rho,r \right]
\end{array} \]
is called an \emph{oracle} associated with $\mathcal{C}$,
$\mathcal{R}$ and $( \left< \,\cdot\,, \,\cdot\, \right>_{(t)} )_{t\geqslant 1}$.
\end{definition}

\begin{definition}[Blackwell's algorithm]
Let $\mathcal{C}$ satisfying Blackwell's
condition (with respect to $\mathcal{R}$ and
$(\left< \,\cdot\,, \,\cdot\, \right>_{(t)} )_{t\geqslant 1}$), $a$
an associated oracle, and $a_1\in \mathcal{A}$. The corresponding \emph{Blackwell's algorithm}
is defined as $\sigma_1=a_1$ and for all $t\geqslant 2$,
$r_1,\dots,r_{t-1}\in \mathbb{R}^d$ and $\rho_1,\dots,\rho_t\in \mathcal{R}$,
\[ \sigma_t(r_1,\dots,r_{t-1},\rho_1,\dots,\rho_{t})=a\left[ t,\rho_t, \sum_{s=1}^{t-1}r_s \right].  \]
\end{definition}
In other words, for a given sequence of outcome functions
$(\rho_t)_{t\geqslant 1}$ in $\mathcal{R}$ and a sequence of Nature's actions
$(b_t)_{t\geqslant 1}$ in $\mathcal{B}$, Blackwell's algorithm
yields the following actions for the Decision Maker:
\[ a_t=a\left[ t,\rho_t,\sum_{s=1}^{t-1}\rho_s(a_s,b_s) \right],\quad t\geqslant 2.  \]
\begin{remark}
For the Blackwell's algorithm to be used by the Decision Maker, the
latter needs to know the outcome function $\rho_t$ and the inner product
$\left< \,\cdot\,, \,\cdot\, \right>_{(t)}$ before choosing its action of stage
$t\geqslant 1$, as the oracle depends on both.
\end{remark}

In the case where the action sets are convex with $\mathcal{A}$ compact, and the
 outcome functions are bi-affine, an equivalent Blackwell's \emph{dual condition} is given below and is often easier to verify, but does not explicitly provide an oracle ---there are however alternative approachability algorithms based on the dual condition~\citep{bernstein2015response}. 
An interesting consequence of the following is that whether a closed convex target set $\mathcal{C}$ satisfies Blackwell's condition does not depend on the sequence of inner products $(\left< \,\cdot\,, \,\cdot\, \right>_{(t)})_{t\geqslant 1}$.

\begin{proposition}[Blackwell's dual condition]
\label{prop:dual}
Assume that $\mathcal{A}$ and $\mathcal{B}$ are convex sets of finite dimensional vectors
spaces, such that $\mathcal{A}$ is compact and all outcome functions $\rho\in \mathcal{R}$ are bi-affine.
Then, a closed convex cone $\mathcal{C}\subset \mathbb{R}^d$ satisfies Blackwell's condition
with respect to $\mathcal{R}$ and $(\left< \,\cdot\,, \,\cdot\, \right>_{(t)})_{t\geqslant 1}$ if, and only if,
\begin{equation}
\label{eq:9}
\tag{$\ast$}
\forall \rho\in \mathcal{R},\ \forall b\in \mathcal{B},\ \exists a\in \mathcal{A},\quad \rho(a,b)\in \mathcal{C}.  
\end{equation}
\end{proposition}
\begin{proof}
  Let $t\geqslant 1$. We first introduce the polar cone of $\mathcal{C}$ associated with
  inner product $\left< \,\cdot\,, \,\cdot\, \right>_{(t)}$:
\[ \mathcal{C}^\circ_{(t)}:=\left\{ z\in \mathbb{R}^d\,\middle|\, \forall r\in \mathcal{C},\ \left< r, z \right>_{(t)}\leqslant 0 \right\}.  \]
If $\mathcal{C}$ is a closed convex cone, an important property is that $(\mathcal{C}^\circ_{(t)})^\circ_{(t)}=\mathcal{C}$,
in other words $r$ belongs to $\mathcal{C}$ if, and only if $\max_{z\in \mathcal{C}^\circ }\left< r, z \right>_{(t)}\leqslant 0$.
Moreover, Moreau's decomposition theorem states that
$r=\pi_{(t)}^{\mathcal{C}}(r)+\pi_{(t)}^{\mathcal{C}^\circ_{(t)}}(r)$ for all $r\in \mathbb{R}^d$. As an immediate consequence, it holds that
$\left\{ r-\pi_{(t)}^{\mathcal{C}}(r) \right\}_{r\in \mathbb{R}^d}=\mathcal{C}^\circ_{(t)}$. 

Then, Blackwell's condition can be written
\[ \forall t\geqslant 1,\ \forall \rho\in \mathcal{R},\quad  \max_{z\in \mathcal{C}^\circ_{(t)} }\min_{a\in \mathcal{A}}\max_{b\in \mathcal{B}}\left< \rho(a,b), z \right>_{(t)} \leqslant 0. \]
Since the quantity $\left< \rho(a,b), z \right>_{(t)}$ is affine in each of the
variables $a$, $b$, and $z$, all three sets are convex, and $\mathcal{A}$ is
compact, we can apply Sion's minimax theorem twice, and equivalently
write:
\[ \forall t\geqslant 1,\ \forall \rho\in \mathcal{R},\quad  \max_{b\in \mathcal{B}}\min_{a\in \mathcal{A}}\max_{z\in \mathcal{C}^\circ_{(t)} }\left< \rho(a,b), z \right>_{(t)} \leqslant 0, \]
which is exactly the dual condition \eqref{eq:9}.
\end{proof}


\section{Analysis}
\label{sec:analysis}
We first give a general guarantee for Blackwell's algorithm. 
\begin{theorem}
\label{thm:approachability-guarantee}
If the family of norms $(\left\|\,\cdot\,\right\|_{(t)})_{t\geqslant 1}$ is nonincreasing (meaning \(\left\|\,\cdot\,\right\|_{(t+1)}\leqslant \left\|\,\cdot\,\right\|_{(t)}\) for all \(t\geqslant 1\)), and $\mathcal{C}$ satisfies Blackwell's
condition (with respect to $\mathcal{R}$ and
$(\left< \,\cdot\,, \,\cdot\, \right>_{(t)} )_{t\geqslant 1}$),
then Blackwell's algorithm, associated with a corresponding oracle, guarantees for all $T\geqslant 1$,
\[ \min_{r\in \mathcal{C}}\left\| \frac{1}{T}\sum_{t=1}^{T}r_t-r \right\|_{(T)}\leqslant \frac{1}{T}\sqrt{\sum_{t=1}^T\left\| r_t \right\|_{(t)}^2}. \]
\end{theorem}
\begin{proof}
  Denote $R_0=0$ and $R_t=\sum_{s=1}^{t}r_s$ ($t\geqslant 1$).
  For all $t\geqslant 1$,
\begin{align*}
\min_{r\in \mathcal{C}}\left\| R_t-r \right\|_{(t)}^2 &=\left\| R_t-\pi_{(t)}^{\mathcal{C}}(R_t) \right\|^2_{(t)}\leqslant \left\| R_t-\pi_{(t)}^{\mathcal{C}}(R_{t-1}) \right\|_{(t)}^2\\
  &=\left\| R_{t-1}+r_t-\pi_{(t)}^{\mathcal{C}}(R_{t-1}) \right\|_{(t)}^2 \\
  &=\left\| R_{t-1}-\pi_{(t)}^{\mathcal{C}}(R_{t-1}) \right\|_{(t)}^2+2\left< R_{t-1}-\pi_{(t)}^{\mathcal{C}}(R_{t-1}), r_t \right>_{(t)} + \left\| r_t \right\|_{(t)}^2.
\end{align*}
First note that for $t=1$, the first two terms of the above last
expression are zero because  $R_{t-1}=\pi_{(t)}^{\mathcal{C}}(R_{t-1})=0$.
For $t\geqslant 2$, we use the nonincreasingness of the family of norms to bound the first term of the above last expression as:
\begin{equation}
\label{eq:1}
\left\| R_{t-1}-\pi_{(t)}^{\mathcal{C}}(R_{t-1}) \right\|_{(t)}^2\leqslant \left\| R_{t-1}-\pi_{(t)}^{\mathcal{C}}(R_{t-1}) \right\|_{(t-1)}^2=\min_{r\in \mathcal{C}}\left\| R_{t-1}-r \right\|_{(t-1)}^2, 
\end{equation}
and the second term (the inner product) is nonpositive because actions
$(a_t)_{t\geqslant 2}$ are given by Blackwell's algorithm, and by denoting
$a$ the involved oracle we get
\[ \left< R_{t-1} -\pi_{(t)}^{\mathcal{C}}(R_{t-1}),r_t \right>_{(t)}=\left< R_{t-1}-\pi_{(t)}^{\mathcal{C}}(R_{t-1}), \rho_t(a\left[ t,\rho_t,R_{t-1} \right], b_t  ) \right>_{(t)},  \]
which is nonpositive by Blackwell's condition.

Therefore, it holds for $t\geqslant 2$ that
$\min_{r\in \mathcal{C}}\left\| R_{t}-r \right\|_{(t)}^2\leqslant \min_{r\in \mathcal{C}}\left\| R_{t-1}-r \right\|_{(t-1)}^2+\left\| r_t \right\|_{(t)}^2$
and $\min_{r\in \mathcal{C}}\left\| R_1-r \right\|_{(1)}^2\leqslant \left\| r_1 \right\|_{(1)}^2$.
Summing and taking the square root gives
\[ \min_{r\in \mathcal{C}}\left\| R_T-r \right\|_{(T)} \leqslant \sqrt{\sum_{t=1}^T\left\| r_t \right\|_{(t)}^2}. \]
We then conclude by remarking that $\mathcal{C}$ being a cone, it
holds that $T\mathcal{C}=\mathcal{C}$ and
\[ \min_{r\in \mathcal{C}}\left\| \frac{1}{T}\sum_{t=1}^Tr_t-r \right\|_{(T)} =\min_{r\in T\mathcal{C}}\frac{1}{T}\left\| R_T-r \right\|_{(T)} =\frac{1}{T}\min_{r\in \mathcal{C}}\left\| R_T-r \right\|_{(T)}. \]
\end{proof}
In the basic setting where the inner product is constant and the outcome vectors are bounded by a given quantity, the above guarantee recovers
the classical $1/\sqrt{T}$ convergence speed.

\begin{remark}
A similar analysis could be carried without the assumption that the norms \((\left\|\,\cdot\,\right\|_{(t)})_{t\geqslant 1}\) are nonincreasing, but inequality (\ref{eq:1}) would not hold and the derived upper bound would contain additional correction terms and would not be as neat as in Theorem~\ref{thm:approachability-guarantee}.
\end{remark}

We assume in the following corollary that the target set is the negative orthant
$\mathcal{C}=\mathbb{R}_-^d$ (which is indeed a closed convex cone)
and present a general choice of time-dependent inner products which yields
different convergence rates for each component of the average
outcome vector. This construction and analysis will be applied
to absorbing games in Sections~\ref{sec:absorbing-games} to~\ref{sec:result-strat-play}.

\begin{corollary}
  \label{cor:1}
Let $(\mu_t^{(1)})_{t\geqslant 1},\dots,(\mu_t^{(d)})_{t\geqslant 1}$
be $d$ positive and nonincreasing sequences, and consider the following
inner products and associated norms:
\[ \left< u, v \right>_{(t)}:=\sum_{i=1}^d(\mu_t^{(i)})^2u^{(i)}v^{(i)},\qquad \left\| u \right\|_{(t)}:=\left< u, u \right>_{(t)}^{1/2}, \qquad u,v\in \mathbb{R}^d,\quad t\geqslant 1. \]
If $\mathcal{C}=\mathbb{R}_-^d$ satisfies Blackwell's condition with respect to
$\mathcal{R}$ and
$(\left< \,\cdot\,, \,\cdot\, \right>_{(t)})_{t\geqslant 1}$, then Blackwell's algorithm
associated with a corresponding oracle guarantees for
each component $1\leqslant i\leqslant d$, and all $T\geqslant 1$,
\[ \frac{1}{T}\sum_{t=1}^Tr_t^{(i)}\leqslant \frac{1}{T\mu_T^{(i)}}\sqrt{\sum_{t=1}^T\left\| r_t \right\|_{(t)}^2}. \]
\end{corollary}
\begin{proof}
Because the sequences
$(\mu_t^{(1)})_{t\geqslant 1},\dots,(\mu_t^{(d)})_{t\geqslant 1}$ are
nonincreasing, so are the associated norms
$(\left\| \,\cdot\, \right\|_{(t)})_{t\geqslant 1}$ 
and Theorem~\ref{thm:approachability-guarantee} applies and gives
\[ \min_{r\in \mathbb{R}_-^d}\left\| \frac{1}{T}\sum_{t=1}^Tr_t - r\right\|_{(T)}\leqslant \frac{1}{T}\sqrt{\sum_{t=1}^T\left\| r_t \right\|_{(t)}^2}. \]
Besides, for each component $1\leqslant i_0\leqslant d$, we can write
\begin{align*}
\min_{r\in \mathbb{R}_-^d}\left\| \frac{1}{T}\sum_{t=1}^Tr_t - r\right\|_{(T)}&=\min_{r^{(1)},\dots,r^{(d)}\leqslant 0}\sqrt{\sum_{i=1}^d(\mu_T^{(i)})^2\left( \frac{1}{T}\sum_{t=1}^T r_t^{(i)}-r^{(i)} \right)^2 }\\
  &\geqslant \min_{r^{(i_0)}\leqslant 0}\mu_T^{(i_0)}\left| \frac{1}{T}\sum_{t=1}^T r_t^{(i_0)}-r^{(i_0)}   \right|  \\
  &\geqslant \frac{\mu_T^{(i_0)}}{T}\sum_{t=1}^Tr_t^{(i_0)}.
\end{align*}
The result follows.
\end{proof}

\section{Absorbing Games}
\label{sec:absorbing-games}

We now apply the above approachability tools to 
 \textit{two-player zero-sum absorbing games}, thereby giving a novel illustration of the relevance of online learning tools for solving games.
We give an alternative proof of a classical result~\citep{kohlberg74,MNR09}: the existence, for each player and each $\varepsilon>0$, of a strategy that is $\varepsilon$-optimal irrespective of the duration of the game, provided that the duration is long enough. 
We first recall in this section the definition of absorbing games and the result we aim at proving, and Sections~\ref{sec:bal} to~\ref{sec:result-strat-play} are dedicated to the proof. Section~\ref{sec:comp} compares our approach to the related literature. 

For a compact or measurable space $A$, we denote by $\Delta(A)$ the set of probability distributions over $A$. If $A$ is compact, $\Delta(A)$ equipped with the weak topology is also compact. 

\subsection{Definition}
A two-player zero-sum absorbing game is described by a tuple $(\Omega, I, J, q, g,g^*)$, where:
\begin{itemize}
    \item  $\Omega=\left\{\Omega^*\right\} \cup \left\{\omega\right\}$ is the \textit{state space}, which is the union of a finite set $\Omega^*$ of \textit{absorbing states}, and $\left\{\omega  \right\} $ containing a unique \textit{non-absorbing state};
    \item $I$ and \(J\) are the \emph{pure action sets} of Player I and Player II, respectively; they are assumed to be compact topological sets; the elements of \(\Delta(I)\) and \(\Delta(J)\) are called \emph{mixed actions};
    \item $q:I \times J \to \Delta(\Omega)$ is the \textit{state transition function}, and is assumed to be separately continuous on $I \times J$; for \(i\in I\), \(j\in J\) and \(\omega\in \Omega\), denote \(q(\omega|i,j)\) the component (probability weight) of \(q(i,j)\) corresponding to \(\omega\);
    \item $g  :I \times J \to \mathbb{R}$ is the \textit{non-absorbing payoff function}, and is assumed to be separately continuous on $I\times J$ and bounded;
    \item $g^*: \Omega \rightarrow \m{R}$ is the \textit{absorbing payoff function}. 
\end{itemize}
Starting from the initial state $\omega_1=\omega$, at each stage $t \geqslant 1$, the game proceeds as follows.
\begin{itemize}
    \item If the current state $\omega_t$ is $\omega$, then simultaneously, Player I (resp.\ Player II) chooses some action $i_t \in I$ (resp.\ $j_t \in J$) possibly drawn according to a mixed action \(x_t\in \Delta(I)\) (resp.\ \(y_t\in \Delta(J)\)). Player I (resp.\ Player II) gets a stage payoff $g( i_t, j_t)$ (resp.\ \(-g( i_t, j_t)\)). A new state $\omega_{t+1}$ is drawn according to distribution $q(i_t,j_t)$, and players observe $(i_t,j_t,\omega_{t+1})$. 
    \item If the current state \(\omega_t\) belongs to $\Omega^*$, then there is no strategic interaction, and Player I gets a stage payoff $g^*(\omega_t)$. The next state is again $\omega_{t+1}=\omega_t$.
\end{itemize}
We denote the stage payoff as
\begin{equation}
\label{eq:stage-payoff}
g_t:=\begin{cases}
    g(i_t,j_t)&\text{if \(\omega_t=\omega\)}\\
    g^*(\omega_t)&\text{if \(\omega_t \in \Omega^*\)}.
  \end{cases} 
\end{equation}
\begin{remark}
Once an absorbing state is reached, the game is essentially over, at least from a strategic point of view. For this reason, we only consider $\omega$ as possible initial state.
 \end{remark}
A sequence $(\omega_1,i_1,j_1,...,\omega_t,i_t,j_t,\dots )\in (\Omega\times I\times J)^{\mathbb{N}^*}$ is called \textit{a play}. 
\begin{itemize}
\item
A \textit{behavior strategy} for a player specifies a mixed action for each possible set of past observations. Formally, for Player I, it is defined as a collection of measurable maps $\sigma=(\sigma_t)_{t \geq 1}$, where $\sigma_t:(I \times J)^{t-1} \rightarrow \Delta(I)$. 
 A behavior strategy $\tau=(\tau_t)_{t \geq 1}$ for Player II is defined similarly.
\item A \textit{Markov strategy}, in the context of absorbing games, is a strategy that plays as a function of the current stage. It can be identified with a sequence of elements in $\Delta(I)$ (resp.\ \(\Delta(J)\)) for Player I (resp.\ Player II). \emph{Pure Markov strategies} are identified with sequences in \(I\) and \(J\).
\item A \textit{stationary strategy}, in the context of absorbing games, is a strategy that plays independently of the past history. It can be identified with an element of $\Delta(I)$ for Player I (resp.\ \(\Delta(J)\) for Player II). 
\end{itemize}
The sets of behavior strategies for Player I and II are denoted by $\Sigma$ and $\mathcal{T}$, respectively. 
A pair of strategies $(\sigma,\tau) \in \Sigma \times \mathcal{T}$ naturally defines, using Kolmogorov's extension theorem, a probability distribution $\mathbb{P}_{\sigma,\tau}$ over the set of plays $(\Omega \times I \times J)^{\mathbb{N}_*}$, 
satisfying, for \((\omega_1,i_1,j_1,\dots,\omega_t,i_t,j_t,\dots )\sim \mathbb{P}_{\sigma,\tau}\), \(\omega_1=\omega\) almost surely, and then for all \(t\geqslant 1\),
\[ h_{t-1}:=(i_1,j_1,\dots,i_{t-1},j_{t-1})\qquad x_t:=\sigma_t(h_{t-1})\qquad y_t:=\sigma_t(h_{t-1}) \]
\[ i_t\mid h_{t-1}\sim x_t\qquad j_t\mid h_{t-1}\sim y_t\qquad i_t\perp j_t\mid h_{t-1}, \]
\[ \omega_{t+1}|i_t,j_t\sim 
\begin{cases}
  \delta_{\omega_t}&\text{if $\omega_t\in \Omega^*$}\\
  q(i_t,j_t)&\text{otherwise}.
\end{cases}
\]
We denote by $\E_{\si,\tau}\left[ \,\cdot\, \right]$ the expectation with respect to \(\mathbb{P}_{\sigma,\tau}\).
 
Let $T \geq 1$ and $\la\in (0,1)$. The \textit{$T$-stage game} (resp.\ $\lambda$-\textit{discounted game}), denoted by $\Gamma_T$ (resp.\ \(\Gamma_{\lambda}\)), is the game with strategy sets $\Sigma$ and $\mathcal{T}$, and payoff function
 \[ 
\ga_T(\sigma,\tau):=\E_{\sigma,\tau} \left[\frac{1}{T} \sum_{t=1}^T g_t \right],\qquad \sigma\in \Sigma,\ \tau\in \mathcal{T}, \]
respectively
\[  \ga_\la(\sigma,\tau):=\E_{\sigma,\tau}\left[\sum_{t \geq 1}\la(1-\la)^{t-1} g_t \right],\qquad \sigma\in \Sigma,\ \tau\in \mathcal{T},\]
where \(g_t\) is defined as in (\ref{eq:stage-payoff}).
\subsection{Values} The games \(\Gamma_{\lambda}\) and \(\Gamma_T\) are known to have values~\citep{MP70}, which are
 \begin{equation*}
v_{T}:=\max_{\sigma \in \Sigma} \min_{\tau \in \mathcal{T}} \gamma_{T}(\sigma,\tau)=\min_{\tau \in \mathcal{T}} \max_{\sigma \in \Sigma} \gamma_{T}(\sigma,\tau)\,.
\end{equation*}
\begin{equation}\label{valSG}
v_{\lambda}:=\max_{\sigma \in \Sigma} \min_{\tau \in \mathcal{T}} \gamma_{\lambda}(\sigma,\tau)=\min_{\tau \in \mathcal{T}} \max_{\sigma \in \Sigma} \gamma_{\lambda}(\sigma,\tau)\,.
\end{equation}

The value $v_\lambda$ can be interpreted as the \textit{payoff solution} of the game $\Gamma_\lambda$: when players play rationally, Player I (resp.\ Player II) should get at least $v_\lambda$ (at most resp.\ $-v_\lambda$). A strategy is \textit{optimal} for Player I (resp.\ for Player II) in $\Gamma_\lambda$ if it reaches the left-hand-side maximum (resp.\ the right-hand side minimum) in \eqref{valSG}. Similar interpretations and definitions hold for $\Gamma_T$. 

In what follows, we consider the linear extension of the non-absorbing payoff function $g$ to $\Delta(I) \times J$, meaning
\[ g(x,j):=\int_I^{}g(i,j)\,\mathrm{d}x(i),\qquad x\in \Delta(I),\ j\in J. \]
Similarly, the transition function $q$ is linearly extended to $\Delta(I) \times J$. 

Optimal strategies can be \textit{a priori} quite sophisticated. 
It turns out that in the discounted game, there exists ``simple'' optimal strategies ~\citep{MP70}, as shown by the following proposition. 
\begin{proposition} \label{prop:shapley}
Let \(\lambda\in (0,1)\).
Each player has an optimal stationary strategy in \(\Gamma_\lambda\). Moreover, if $x_{\lambda} \in \Delta(I)$ is an optimal stationary strategy for Player I, then for all $j \in J$,
\begin{equation*}
v_\lambda \leq \lambda g(x_{\lambda},j) +(1-\lambda)\left( \sum_{\omega^* \in \Omega^*} q(\omega^*|x_{\lambda},j) g^*(\omega^*)
+q(\omega|x_{\lambda},j)v_\lambda \right).
\end{equation*}
\end{proposition}
The above inequality stems from the fact that $v_\lambda$ satisfies a functional equation called the \textit{Shapley equation} (we refer the reader to \citet{shapley1953stochastic} when the state space and the action sets are finite, and \citet{MP70} for the general case). 

Let us intuitively explain the above inequality. Assume that in state $\omega$ in the discounted game $\Gamma_\lambda$, Player I plays the stationary strategy $x_{\lambda}$ at every stage. Assume moreover that Player II plays pure action $j$ at the first stage, and then plays optimally in $\Gamma_\lambda$ from stage 2. At stage 1, the expectation of the stage payoff is $g(x_{\lambda},j)$, and the next state is distributed according to $q(x_{\lambda},j)$. If the next state is $\omega^* \in \Omega^*$, then the total payoff from the next stage is $\sum_{t \geq 2} \lambda (1-\lambda)^{t-1} g^*(\omega^*)=(1-\lambda)g^*(\omega^*)$. If the next state is $\omega$, since both players play optimally from stage 2, the payoff from stage 2 is equal to $(1-\lambda) v_\lambda$. Hence, the right-hand side of the inequality is exactly what players get in $\Gamma_\lambda$. Since $x_{\lambda}$ is optimal, this quantity should be higher than $v_\lambda$, and this yields the inequality. 

\begin{definition} \label{def:uniform}
Let \(w\in \mathbb{R}\). 
\begin{enumerate}[(i)]
\item Player I can \textit{uniformly guarantee} $w$ if for all $\varepsilon>0$, there exists $\sigma \in \Sigma$ and $T_0 \geq 1$ such that for all $T \geq T_0$ and $\tau \in \mathcal{T}$, $\gamma_T(\sigma,\tau) \geq w-\varepsilon$.
\item Similarly, Player II can \textit{uniformly guarantee} $w$ if for each $\varepsilon>0$, there exists $\tau \in \mathcal{T}$ and $T_0 \geq 1$ such that for all $T \geq T_0$ and $\sigma \in \Sigma$, $\gamma_T(\sigma,\tau) \leq w+\varepsilon$.
\end{enumerate}
\end{definition}
\citet{MNR09} have proved the following result.
\begin{theorem} \label{theo:main}
There exists a unique $v \in \m{R}$ such that both players can uniformly guarantee $v$. Moreover,
\[ v=\lim_{T\to + \infty}v_T=\lim_{\lambda\to 0^+}v_{\lambda}. \]
\end{theorem}
This generalizes a result of \citet{kohlberg74}, that treated the case of finite action sets. Another proof of Theorem~\ref{theo:main} was also given in \citet{HIR21}.
\begin{remark} \label{rem:conv}
The uniqueness of $v$ follows from Definition~\ref{def:uniform}, and $v$ is called the \textit{uniform value}. The fact that the existence of $v$ implies the convergence of $v_T$ and $v_\lambda$ to $v$ is rather straightforward, and can be found for instance in \cite[Chapter 4]{sorin02b}. This limit \(v\) is called the \emph{limit value}. 
\end{remark}
Our goal is to give a proof of Theorem~\ref{theo:main} based on Blackwell's approachability. We will prove that Player I can uniformly guarantee $\limsup _{\lambda \rightarrow 0^+} v_\lambda$. By switching players' roles, we then immediately deduce that Player II can uniformly guarantee $\liminf_{\lambda \rightarrow 0^+} v_\lambda$, hence can uniformly guarantee $\limsup _{\lambda \rightarrow 0^+} v_\lambda$. Together with Remark~\ref{rem:conv}, this proves the theorem. The reader can find a comparison between this new proof and the proofs from \citet{kohlberg74,MNR09,HIR21} in Section~\ref{sec:comp}. 

Let \(\varepsilon\in (0,1/2)\) be fixed. The remaining of the paper is devoted to constructing a strategy for Player I that guarantees \(\limsup_{\lambda \rightarrow 0^+}v_{\lambda}-O(\varepsilon)\) in \(\Gamma_T\) for large enough \(T\).
Up to translating and multiplying the payoffs by a scalar, we can assume the following without loss of generality.
\begin{assumption}
    \label{assumption}
\[ \limsup\limits_{\lambda \to 0^+} \, v_\lambda = \varepsilon,\qquad \left\| g \right\|_{\infty}\leqslant 1\quad \text{and}\quad \left\| g^* \right\|_{\infty}\leqslant 1. \]
\end{assumption}
\section{Balanced Strategies and Lower Bounds on the Average Payoff} \label{sec:bal}
\label{sec:anaysis_absorbing}
For each $\lambda\in (0,1)$, let $x_\lambda\in \Delta(I)$ be an optimal stationary strategy for Player I. We consider a sequence $(\lambda_k)_{k\geqslant 1}$ such that
\begin{equation}
\label{eq:lambda-k}
v_{\lambda_k} \xrightarrow[k \rightarrow + \infty]{}\limsup_{\lambda\to 0^+} v_\lambda\quad \text{and}\quad \text{$x_{\lambda_k}$ converges to some $x_0 \in \Delta(I)$.} 
\end{equation}
Thanks to Assumption~\ref{assumption} and the fact that $g$ is separately continuous, there exists $\lambda \in (0,1)$ such that
the following holds.

\bigskip 

\begin{quote}
\emph{For the remaining of the paper, let \(\lambda\in (0,1)\) be such that     
for all $j \in J$,
	\begin{equation} \label{eq:lambda}
	|g(x_0,j) - g(x_{\lambda},j)| \leqslant \varepsilon/2\qquad \text{and}\qquad  v_\lambda \geqslant \varepsilon/2.
	\end{equation}
 }
\end{quote}

\bigskip

\begin{definition} \label{def:balanced}
    A strategy $(\sigma_t)_{t\geqslant 1}$ of Player I is said to be \textit{balanced} (with respect to \(x_0\) and \(x_{\lambda}\)) if for any $t \geq 1$ and history $h_{t-1}=(i_1,j_1,\dots,i_{t-1},j_{t-1})\in (I\times J)^{t-1}$,
\begin{enumerate}[(i)] 
\item $\sigma_t(h_{t-1})$ only depends on $(j_1,\dots,j_{t-1})$; \label{def:bal1}
\item $\sigma_t(h_{t-1})$ is a convex combination of $x_0$ and $x_\lambda$.
\end{enumerate}
\end{definition}
A first important point is that, in order to prove that a balanced strategy guarantees some quantity, it is sufficient to verify that it guarantees that quantity against all pure Markov strategies of Player II. The proof is quite standard and is given in Section~\ref{sec:pure-markov} for completeness. 
\begin{lemma}
\label{lm:pure-markov}
Let $\sigma\in \Sigma$ be a strategy that satisfies (\ref{def:bal1}) in Definition \ref{def:balanced}. 
Let $w \in \m{R}$ and $T \geq 1$. 
Assume that for any pure Markov strategy $\tau=(j_t)_{t \geq 1}$ of Player II, 
$\gamma_T(\sigma,\tau) \geq w$. Then for any $\tau \in \mathcal{T}$, $\gamma_T(\sigma,\tau) \geq w$. 
\end{lemma}
\bigskip

\begin{quote}
  \emph{For the remaining of the paper,
\begin{itemize}
\item let \(\sigma\) be a balanced strategy for Player I with respect to \(x_0\) and \(x_{\lambda}\),
\item \(\tau=(j_t)_{t\geqslant 1}\) be a pure Markov strategy for Player II,
\item and for each \(t\geqslant 1\), let \(a_t\in \left[ 0,1 \right] \) be the weight on mixed action \(x_\lambda\) when \(\sigma\) is played against \(\tau\), in other words:
\begin{equation}
\label{eq:a-t}
\sigma_t(j_1,\dots,j_{t-1})=a_tx_{\lambda}+(1-a_t)x_0.  
\end{equation}
\end{itemize}  }
\end{quote}

\bigskip 


We now introduce three quantities that play a crucial role in our construction. For all $j \in J$, let
\[
g^\sharp(j) := \sum_{\omega^* \in \Omega^*} g^*(\omega^*)\, q(\omega^*| x_\lambda, j) \quad \text{and} \quad g^\flat(j) := g(x_0, j).
\]
The quantity $g^\sharp(j)$ corresponds to the future expected payoff (induced by $x_\lambda$ and $j$) in the case where the next state is absorbed, whereas the quantity 
 $g^\flat(j)$ is the present payoff in \(\omega\), which is close to the payoff given by $x_\lambda$ and $j$, thanks to \eqref{eq:lambda}. 

The case where \(g^\flat(j)=g(x_0,j)\geqslant 0\) for all \(j\in J\) is easy, because the strategy for Player I that plays mixed action \(x_0\) at each stage obviously guarantees \(0\) in \(\Gamma_T\) for all \(T\geqslant 1\), which reaches our goal. Therefore, we will consider only the following case.
\begin{assumption}
    \label{ass:not-easy}
    There exists \(j\in J\) such that \(g^\flat (j)<0\).
\end{assumption}


The following lemma gives an expression of the probability that the game is in state $\omega$ at stage \(t\). The proof in given in Appendix~\ref{sec:alpha}.
\begin{lemma}
    \label{lm:alpha}
Let
\begin{equation}
\label{eq:alpha}
\alpha_t := \prod_{s=1}^{t-1} \big( a_s q(\omega|x_\lambda, j_s) + (1-a_s)q(\omega|x_0,j_s) \big), \qquad t\geqslant 1.
\end{equation}
Then, \(\mathbb{P}_{\sigma,\tau}\left[ \omega_t=\omega \right]=\alpha_t\).
\end{lemma}
We are now ready to state a key lower bound on the payoff \(\gamma_T(\sigma,\tau)\) that only involves the quantities $g^\sharp(j)$, $g^\flat(j)$, $\alpha_s$ and $a_s$. 
\begin{proposition} \label{prop:payoff_bound}
For all \(T\geqslant 1\),
\begin{equation}\label{eq:lower_bound_a_b}
\gamma_T(\sigma, \tau) \geqslant \frac{1}{T} \sum_{t=1}^T \alpha_t g^\flat(j_t) + \sum_{t=1}^T \sum_{s=1}^{t-1}  \alpha_s a_s g^\sharp(j_s)- \varepsilon.
\end{equation}
\end{proposition}

We first prove the following lemma and then establish Proposition~\ref{prop:payoff_bound}. 
\begin{lemma}\label{lm:x_0}
Let $j \in J$. The following statements hold.
\begin{enumerate}[(i)]
\item \label{sign}
	\[
    \lambda g^\flat(j) + (1-\lambda) g^\sharp(j) \geqslant 0.
    \] 
In particular, either $g^\sharp(j) \geqslant 0$ or $g^\flat(j) \geqslant 0$.
\item \label{positive_value}
    \[
    \sum_{\omega^* \in \Omega^*}  g^*(\omega^*)\,q(\omega^* |\,x_0, j)\geqslant 0.
    \]
 \end{enumerate}
\end{lemma}
\begin{proof}
\begin{enumerate}[(i)]
\item
Using \eqref{eq:lambda} and Proposition~\ref{prop:shapley}, we get
	\begin{align*}
	\lambda g^\flat(j) + (1-\lambda) g^\sharp(j)&=\lambda\, g(x_0, j) + (1-\lambda) \sum_{\omega^* \in \Omega^*} g^*(\omega^*) q(\omega^*| x_{\lambda}, j)   \\
&\geqslant \lambda\, g(x_\lambda, j) + (1-\lambda) \sum_{\omega^* \in \Omega^*} g^*(\omega^*)\, q(\omega^*| x_{\lambda}, j) -\lambda\varepsilon/2  \\ 
&\geqslant v_\lambda \big(1 -  (1-\lambda)q(\omega| x_\lambda, j)\big) -\lambda\varepsilon/2  \\ 
&\geqslant \frac{\varepsilon}{2} \big(1 -  (1-\lambda)\big) -\lambda\varepsilon/2 \\
&\geqslant 0.
	\end{align*}
\item
Using Assumption~\ref{assumption}, Proposition~\ref{prop:shapley}, the continuity of the transition kernel $q$ with respect to $i\in I$ (thus continuity with respect to the weak topology on the space of mixed actions), and a sequence \((\lambda_k)_{k\geqslant 1}\) satisfying (\ref{eq:lambda-k}),
	\begin{align*}
		\sum_{\omega^* \in \Omega^*}  g^*(\omega^*) q(\omega^*| x_0, j)  
		&= \lim_{k \rightarrow +\infty} \left( \lambda_k g(x_{\lambda_k}, j) + (1-\lambda_k) \sum_{\omega^* \in \Omega^*} g^*(\omega^*) q(\omega^*| x_{\lambda_k}, j) \right) {\color{red} } \\
		&\geqslant  \lim_{k \rightarrow+\infty} v_{\lambda_k}(1 -  (1-\lambda_k)q(\omega| x_{\lambda_k}, j)) \geqslant 0.
	\end{align*}
\end{enumerate}
\end{proof}

\begin{proof}[Proof of Proposition~\ref{prop:payoff_bound}]
We have
\begin{align*}
\gamma_T(\sigma, \tau) &= 
  \frac{1}{T}\expec{\sigma,\tau} {\sum_{t=1}^T \mathbbm{1}_{\{\omega_t = \omega\}} g(i_t,j_t)}
  +  \frac{1}{T}\expec{\sigma,\tau} {\sum_{t=1}^T \sum_{\omega^* \in \Omega^*} \mathbbm{1}_{\{\omega_t = \omega^*\}} g^*(\omega^*)} \\
=& 
 \frac{1}{T}\expec{\sigma,\tau}{ \sum_{t=1}^T \mathbbm{1}_{\{\omega_t = \omega\}} g(i_t,j_t) + \sum_{t=1}^T \sum_{s=1}^{t-1} \mathbbm{1}_{\{\omega_s = \omega\}} \sum_{\omega^* \in \Omega^*} g^*(\omega^*) q(\omega^*\,|\,i_s, j_s)}.
\end{align*}
Let us bound from below the term inside the first sum. We have
\begin{align*}
\expec{\sigma,\tau}{(\mathbbm{1}_{\{\omega_t=\omega\}} g(i_t,j_t)}&=  \m{P}_{\sigma,\tau}\left[  \omega_t=\omega \right]
\left(a_t g(x_\lambda,j_t) + (1-a_t)g(x_0,j_t)  \right)\\
 &\geqslant \alpha_t\left( a_t(g(x_0,j_t)-\varepsilon/2)+(1-a_t)g(x_0,j_t) \right)\\  
&\geq \alpha_t g^\flat(j_t)-\varepsilon,
\end{align*}
where the first equality stems from the independence of $\omega_t$ and $i_t$, and the second inequality comes from \eqref{eq:lambda}. 
Using the same independence and Property~\eqref{positive_value} from Lemma~\ref{lm:x_0}, we get
\begin{multline*}
\mathbb{E}_{\sigma,\tau}\left[ \mathbbm{1}_{\{\omega_s = \omega\}}\sum_{\omega^* \in \Omega^*} g^*(\omega^*) q(\omega^*|i_s,j_s) \right]\\=
\alpha_s \sum_{\omega^* \in \Omega^*} g^*(\omega^*) \left( a_s q(\omega^*|x_\lambda,j_s)+(1-a_s)q(\omega^*|x_0,j_s) \right) 
\\\geq \alpha_s a_s g^\sharp(j_s).\qquad \qquad \qquad \qquad \qquad \qquad \qquad \qquad \qquad \qquad \qquad \quad 
\end{multline*}
The result follows from plugging the two previous inequalities into the first equality.
\end{proof}
Now that the proof of Proposition~\ref{prop:payoff_bound} is complete, let us explain at a broad level how we can relate the lower bound to the approachability tools that we developed in the first part of the paper. The lower bound \eqref{eq:lower_bound_a_b} involves two sums. A first idea would be to consider a two-dimensional vector payoff corresponding to the two terms involved in the sums, namely $(-\alpha_s g^\flat(j_s), -\alpha_s a_s g^\sharp(j_s))$, and show that Player I can approach the set $\m{R}^2_{-}$ for this vector payoff function. This would hopefully prove that both sums are larger than some $O(\varepsilon)$, when $T$ is large enough. This simple idea does not work, for two reasons. First, the component $-\alpha_s g^\flat(j_s)$ does not depend on $a_s$, hence there is no hope that Player I can force this component to be negative, even on average. Moreover, the second sum is a \textit{double sum}. As a consequence, a single quantity $a_s$ can have a contribution to the average payoff that does not vanish as $T\to + \infty$. Hence, we in fact need all $a_s$ to be very small, which motivates the introduction of a scaling factor that multiplies the $a_s$. These considerations lead to the following modified lower bound. 

\begin{proposition}
\label{prop:payoff_bound2}
Let $(\mu_t)_{t\geqslant 1}$ be a positive and nonincreasing sequence
and \((a_t)_{t\geqslant 1}\) defined as in (\ref{eq:a-t}). If
\(a_t\in \left[ 0,\mu_t \right]\) for all \(t\geqslant 1\), 
 then for all $T \geqslant 4/(\lambda \varepsilon^2 \mu_T)$,
    \begin{equation} \label{eq:lb2}
    \gamma_T(\sigma,\tau) \geqslant 
\frac{1}{T} \left(\sum_{t=1}^T \alpha_t\left(1 - \mu_t^{-1}a_t\right) g^\flat(j_t) + \sum_{t=1}^T \sum_{s=1}^{t-1} \alpha_s a_s g^\sharp(j_s) \right)\\ - 2 \varepsilon.
 \end{equation}
\end{proposition}
\begin{proof}
From Proposition~\ref{prop:payoff_bound}, we have
\begin{align*}
\gamma_T(\sigma, \tau) &\geqslant \frac{1}{T} \left(\sum_{t=1}^T \alpha_t g^\flat(j_t) + \sum_{t=1}^T \sum_{s=1}^{t-1}  \alpha_s a_s g^\sharp(j_s) \right) - \varepsilon
\\
&=
\frac{1}{T} \left(\sum_{t=1}^T \alpha_t\left(1 - \mu_t^{-1}a_t\right) g^\flat(j_t) + \sum_{t=1}^T \sum_{s=1}^{t-1} \alpha_s a_s g^\sharp(j_s) \right)
\\
  &\qquad \qquad \qquad +\frac{1}{T}\sum_{t=1}^T \mu^{-1}_t \alpha_t a_t\, g^\flat(j_t)-\varepsilon.
\end{align*}
To prove Proposition~\ref{prop:payoff_bound2}, it is enough to show that for $T \geqslant 4/(\lambda \varepsilon^2 \mu_T)$, the above right-most average is higher than $-\varepsilon$.

Denote $q^*(j) :=  1-q(\omega\,|\,x_{\lambda}, j)$, which is the probability that the game is absorbed when Player I plays $x_\lambda$ and Player II plays $j$, and 
%
$J(\varepsilon) := \{j\in J: q^*(j) \leqslant \lambda \varepsilon/2 \}$. 
Let $j \in J(\varepsilon)$. Using Proposition~\ref{prop:shapley} and the fact that \(g^*\leqslant 1\),
\begin{align*}
\lambda g(x_{\lambda}, j)
& \geqslant  v_{\lambda} \big(1 -  (1-{\lambda})q(\omega| x_{\lambda}, j)\big) - (1-\lambda) \sum_{\omega^* \in \Omega^*} g^*(\omega^*) q(\omega^*| x_\lambda, j) \\
& \geqslant  v_{\lambda} \big(1 -  (1-{\lambda})(1-q^*(j))\big) - (1-\lambda) \sum_{\omega^* \in \Omega^*} q(\omega^*| x_\lambda, j) \\
&= v_{\lambda} \big(\lambda+(1-{\lambda})q^*(j)\big) - (1-\lambda)q^*(j) \\
&= \lambda v_{\lambda} + (1-\lambda) q^*(j) (v_{\lambda}-1) \\
&\geqslant \frac{\lambda\varepsilon}{2} - \frac{\lambda \varepsilon}{2} = 0,
\end{align*}
where the last inequality follows from \eqref{eq:lambda} and the fact that \(j\in J(\varepsilon)\). Hence $g(x_{\lambda}, j) \geqslant 0$, and we deduce that
\[
 \frac{1}{T} \sum_{t=1}^T \alpha_t \mu^{-1}_t a_t g(x_{\lambda}, j_t) \geqslant \frac{1}{T} \sum_{\substack{1\leqslant t\leqslant T\\j_t \notin J(\varepsilon)}}^{} \alpha_t \mu^{-1}_t a_t g(x_\lambda, j_t).
\]
Let $t\geqslant 1$ such that $j_t \in J \setminus J(\varepsilon)$. We have
\begin{align*}
\alpha_{t+1} 
&= \big( a_t q(\omega|x_\lambda,j_t) + (1-a_t) q(\omega|x_0,j_t) \big) \alpha_t \\
&\leqslant \big( a_t(1-q^*(j_t)) + (1-a_t)\big) \alpha_t  \leqslant \left( 1 - \frac{\lambda\varepsilon}{2}a_t \right) \alpha_t,
\end{align*}
from which we deduce that for all $t \geq 1$,
\begin{equation*}
\alpha_t \leq \prod_{\substack{1\leqslant s<t\\j_s \notin J(\varepsilon)}} \left( 1 - \frac{\lambda\varepsilon}{2}a_s \right).
\end{equation*}
Combining with the fact that $g \geq -1$ and $\mu_t$ is non-increasing, we obtain
\begin{align*}
\sum_{\substack{1\leqslant t \leqslant T\\j_t \notin J(\varepsilon)}} \alpha_t \mu^{-1}_t a_t g(x_\lambda, j_t) &\geqslant - \mu^{-1}_T\sum_{\substack{1\leqslant t\leqslant T\\j_t \notin J(\varepsilon)}} \alpha_t  a_t  \\
& \geqslant - \mu^{-1}_T \sum_{\substack{1\leqslant t\leqslant T\\j_t \notin J(\varepsilon)}} a_t \prod_{\substack{1\leqslant s<t\\j_s \notin J(\varepsilon)}} \left( 1 - \frac{\lambda\varepsilon}{2}a_s \right).
\end{align*}
Define $p_s:=0$ if $j_s \in J(\varepsilon)$, and $p_s:=\lambda\varepsilon a_s/2$ otherwise, so that 
\[ 
\sum_{\substack{1\leqslant t\leqslant T\\j_t \notin J(\varepsilon)}} a_t \prod_{\substack{1\leqslant s<t\\j_s \notin J(\varepsilon)}} \left( 1 - \frac{\lambda\varepsilon}{2}a_s \right)=\frac{2}{\lambda \varepsilon} \sum_{t=1}^\infty p_t \prod_{s=1}^{t-1} \left( 1 - p_s \right).
 \]
Let $X_1, X_2, \dots$ be a sequence of independent Bernoulli random variables with $\mathbb{P}[X_t=1] = p_t$. Then
\[
\sum_{t=1}^\infty p_t \prod_{s=1}^{t-1} \left( 1 - p_s \right) = \sum_{t=1}^{\infty} \mathbb{P}[X_1, \dots, X_{t-1}=0, X_t=1] = \mathbb{P}[\exists\,t,\, X_t=1] \le 1.\]
We deduce that
\[
\sum_{\substack{1\leqslant t\leqslant T\\j_t \notin J(\varepsilon)}} a_t \prod_{\substack{1\leqslant s<t\\j_s \notin J(\varepsilon)}} \left( 1 - \frac{\lambda\varepsilon}{2}a_s \right)  	 \leqslant \frac{2}{\lambda \varepsilon}.
\]
Therefore if $T \geqslant 4/(\lambda \varepsilon^2 \mu_T)$,
\begin{align*}
\frac{1}{T} \sum_{t=1}^T \alpha_t \mu^{-1}_t a_t g(x_{\lambda}, j_t) \geqslant
& - \frac{2}{\lambda \varepsilon\mu_T T} \geqslant -\frac{\varepsilon}{2}.
\end{align*}
Using \eqref{eq:lambda} and the fact that \(a_t\leqslant \mu_t\) by assumption,
we deduce that
\[
 \frac{1}{T} \sum_{t=1}^T \mu^{-1}_t \alpha_t a_t g(x_0, j_t)\geqslant - \varepsilon,
\]
which completes the proof.
\end{proof}
In the next section, we apply the approachability tools from the first section to a two-dimensional vector payoff, composed of the two terms that appear in each sum of the lower bound from Proposition~\ref{prop:payoff_bound2}. 

\section{An Auxiliary Approachability Problem}
\label{sec:an-auxil-appr}

We consider the sets of actions
$\mathcal{A}=[0,1]$, $\mathcal{B} = J$, and target set
$\mathcal{C}=\mathbb{R}_-^2$. For $\alpha\geqslant 0$ and $\mu >0$, denote
\[ \rho_{\alpha,\mu}:(a,j) \mapsto \begin{pmatrix}
 - \alpha a g^\sharp(j)\\
\alpha(\mu^{-1} a-1) g^\flat(j)
 \end{pmatrix},\]
and let $\mathcal{R}=\left\{ \rho_{\alpha,\mu} \right\}_{\substack{\alpha\geqslant 0 \\\mu\in (0,1]}}$ be the set of possible outcome functions. Let $(\mu_t)_{t\geqslant 1}$ be a positive and nonincreasing
sequence in $(0,1]$,
and define the following inner product on $\mathbb{R}^2$:
\[ \left< u, v \right>_{(t)} =u^{(1)}v^{(1)}+\mu_t^2 u^{(2)}v^{(2)},\qquad u,v\in \mathbb{R}^{2},\quad t\geqslant 1. \]
For \(R\in \mathbb{R}^2\), define
\[ \pi(R):=(\max_{}(0,R^{(1)}),\max_{}(0,R^{(2)})). \]
\begin{proposition}[Blackwell's condition]\label{prop:oracle}
Let $t\geqslant 1$, $\alpha \geqslant 0$, $\mu \in (0,1]$ and $R = (R^{(1)}, R^{(2)}) \in \mathbb{R}^2$.
Denote the components of \(\pi (R)\) as \((\tilde{R}^{(1)},\tilde{R}^{(2)})\).
\begin{enumerate}[(i)]
  \item\label{item:projection} \(\pi(R)=R-\argmin_{R'\in \mathbb{R}_-^2}\left\| R'-R \right\|_{(t)}\).
 \item\label{item:bounds-action} The following quantity
  \[ a [t,\rho_{\alpha,\mu},R]:=
    \begin{cases}
    0&\text{if $\pi(R)=0$}  \\
\displaystyle \sup_{\substack{j\in J\\
      g^\flat(j)<0}}\frac{\mu_t^2 g^\flat(j)\tilde{R}^{(2)} }{\mu^2_t \mu^{-1} g^\flat(j)\tilde{R}^{(2)} - g^\sharp(j)\tilde{R}^{(1)}} &\text{otherwise,}
    \end{cases}
     \] 
  is well-defined and satisfies $0\leqslant a [t,\rho_{\alpha,\mu},R]\leqslant \mu \leqslant 1$;
\item\label{item:oracle} For all $j \in J$ and $R \in \mathbb{R}^2$,
\[ \left< \rho_{\alpha,\mu}(a [t,\rho_{\alpha,\mu},R],j), \pi(R) \right>_{(t)}\leqslant 0. \]
\end{enumerate}
In other words, $\mathcal{C}=\mathbb{R}_-^2$ satisfies Blackwell's
  condition with respect to
  $\mathcal{R}$ and
  $(\left< \,\cdot\,, \,\cdot\, \right>_{(t)})_{t\geqslant 1}$, with mapping $(t,\rho,R)\mapsto a [t,\rho,R]$ being an associated oracle.
\end{proposition}
\begin{proof}
We will use repeatedly the fact that \(\tilde{R}^{(1)}\) and \(\tilde{R}^{(2)}\) are nonnegative by definition.
\begin{enumerate}[(i)]
\item $\pi(R)$ can be written as
\begin{align*}
\pi(R) &= R-\argmin_{R'\in \mathbb{R}_-^2}\left\| R'-R \right\|_{(t)}^2\\
  &=R-\argmin_{R'^{(1)},R'^{(2)}\leqslant 0}\left\{ (R'^{(1)}-R^{(1)})^2+\mu_t^2(R'^{(2)}-R^{(2)})^2 \right\}.
\end{align*}
Therefore, for each component $i\in \left\{ 1,2 \right\}$,
\[ \tilde{R}^{(i)}=R^{(i)} - \argmin_{R'^{(i)}\leqslant 0}(R'^{(i)}-R^{(i)})^2=R^{(i)}-\min_{}(0,R^{(i)})=\max_{}(R^{(i)},0)\geqslant 0. \]
\item $a [t,\rho_{\alpha,\mu},R]$ is either zero or defined as the supremum of a 
  fraction. In the latter case, meaning \(\pi(R)\neq 0\), first note that the supremum is taken on a nonempty set thanks to Assumption~\ref{ass:not-easy}. Then, for \(j\in J\) such that $g^\flat(j)<0$, we have $g^\sharp(j)>0$ by Property~\eqref{sign} in Lemma~\ref{lm:x_0}, and $\mu$ is positive from the statement. Therefore, both the numerator is nonpositive and the denominator is negative, and thus
  \[
 0\leqslant \frac{\mu_t^2 g^\flat(j) \tilde{R}^{(2)} }{\mu_t^2\mu^{-1}g^\flat(j)\tilde{R}^{(2)} - g^\sharp(j)\tilde{R}^{(1)} } \leqslant  \frac{\mu_t^2 g^\flat(j) \tilde{R}^{(2)} }{\mu_t^2\mu^{-1}g^\flat(j)\tilde{R}^{(2)}} = \mu.
  \]
\item Let \(j\in J\) be fixed and denote \(g^{\sharp}:=g^{\sharp}(j)\) and \(g^{\flat}:=g^{\flat}(j)\) for short. For every $a\in [0,\mu]$, the following inner product writes
\begin{align*}
\left< \rho_{\alpha,\mu}(a,j), \pi(R) \right>_{(t)} 
&=- \alpha \left(  a g^{\sharp} \tilde{R}^{(1)} + \mu_t^{2}( 1-\mu^{-1}a) g^{\flat} \tilde{R}^{(2)} \right) \\
&=- \alpha \left( a \left( g^{\sharp}\tilde{R}^{(1)} -  \mu^2_t\mu^{-1}g^{\flat}\tilde{R}^{(2)}\right) + \mu_t^2 g^{\flat}\tilde{R}^{(2)} \right).
\end{align*}
If $\pi(R)=0$, the result is true for any choice of $a$.

Otherwise, if $g^{\flat}=0$, it follows that $g^{\sharp}\geqslant 0$ by Property~\eqref{sign} in Lemma~\ref{lm:x_0}, and 
\[\left< \rho_{\alpha,\mu}(a,j), \pi(R) \right>_{(t)}  =- \alpha\, a\, g^{\sharp} \tilde{R}^{(1)} \le 0. \] 

If  $g^{\flat} > 0$ and $g^{\sharp} \geqslant 0$, since $a\leqslant \mu$, we have
\[
\left< \rho_{\alpha,\mu}(a,j), \pi(R) \right>_{(t)} = - \alpha \left( a g^{\sharp} \tilde{R}^{(1)} + \mu_t^{2}(1 -  \mu^{-1}a) g^{\flat} \tilde{R}^{(2)} \right)\leqslant 0,
\]
because all quantities in the above grand parentheses are nonnegative.

If $g^{\flat} < 0$, which implies that $g^{\sharp}>0$, we have \(g^{\sharp}\tilde{R}^{(1)}-\mu_t^2\mu^{-1}g^{\flat}\tilde{R}^{(2)}>0\) and therefore
\begin{align*}
&\left< \rho_{\alpha,\mu}(a[t,\rho_{\alpha,\mu},R],j), \pi(R) \right>_{(t)}\\
&	\qquad = - \alpha \left( a[t,\rho_{\alpha,\mu},R] \cdot \left( g^{\sharp}\tilde{R}^{(1)} - \mu^2_t\mu^{-1}g^{\flat}\tilde{R}^{(2)}\right) + \mu_t^2 g^{\flat}\tilde{R}^{(2)} \right)\\
	&\qquad =- \alpha \Biggl( \left( g^{\sharp}\tilde{R}^{(1)} - \mu^2_t\mu^{-1}g^{\flat}\tilde{R}^{(2)}\right)\sup_{\substack{j'\in J\\g^{\flat}(j')<0}}\frac{-\mu_t^2g^{\flat}(j')\tilde{R}^{(2)}}{g^{\sharp}(j')\tilde{R}^{(1)}-\mu_t^2\mu^{-1}g^{\flat}(j')\tilde{R}^{(2)}} \Biggr.\\
&\qquad \qquad \qquad \qquad \qquad \qquad \qquad\qquad \qquad \qquad \qquad  \Biggl.+ \mu_t^2g^{\flat}\tilde{R}^{(2)}  \Biggr) \\
 &\qquad \leqslant - \alpha\left( \left( g^{\sharp}\tilde{R}^{(1)} - \mu^2_t\mu^{-1}g^{\flat}\tilde{R}^{(2)}\right)\frac{-\mu_t^2g^{\flat}\tilde{R}^{(2)}}{g^{\sharp}\tilde{R}^{(1)}-\mu_t^2\mu^{-1}g^{\flat}\tilde{R}^{(2)}} + \mu_t^2g^{\flat}\tilde{R}^{(2)}\right)\\
 &\qquad \leqslant 0.
\end{align*}

We now turn to the last remaining case and assume $g^{\flat}>0$ and $g^{\sharp}\leqslant 0$.
Then, note that \(\mu^2_t\mu^{-1}g^{\flat}\tilde{R}^{(2)}-g^{\sharp}\tilde{R}^{(1)} \geqslant 0\).
Property~\eqref{item:oracle} is equivalent to
\[a[t,\rho_{\alpha,\mu},R] \left( \mu^2_t\mu^{-1}g^{\flat}\tilde{R}^{(2)}-g^{\sharp}\tilde{R}^{(1)}\right) - \mu_t^2 g^{\flat}\tilde{R}^{(2)}\leqslant 0. \]
If  \(\mu^2_t\mu^{-1}g^{\flat}\tilde{R}^{(2)}-g^{\sharp}\tilde{R}^{(1)} =0\) or \(\tilde{R}^{(2)}=0\), the property is easily satisfied (because \(\tilde{R}^{(2)}=0\) implies \(a\left[ t,\rho_{\alpha,\mu},R \right]=0 \)).
We now assume those two quantities to be positive. Then, the property is equivalent to
\[
    a\left[ t,\rho_{\alpha,\mu},R \right]=
\sup_{\substack{j'\in J\\g^{\flat}(j')<0}} \frac{\mu_t^2 g^{\flat}(j')\tilde{R}^{(2)} }{\mu^2_t \mu^{-1} g^{\flat}\tilde{R}^{(2)} - g^{\sharp}\tilde{R}^{(1)}} \leqslant \frac{\mu^2_t g^{\flat}\tilde{R}^{(2)}}{\mu^2_t \mu^{-1} g^{\flat}\tilde{R}^{(2)} - g^{\sharp}\tilde{R}^{(1)}},
\]
which, after simplification, is equivalent to
\[
\sup_{\substack{j'\in J\\g^{\flat}(j')<0}} \frac{g^{\sharp}(j')}{g^{\flat}(j')} \leqslant \frac{g^{\sharp}}{g^{\flat}},
\]
which we now aim at proving. Let \(j'\in J\) such that \(g^{\flat}(j')<0\).
Then Property~\eqref{sign} from Lemma~\ref{lm:x_0} applied to \(j\) and \(j'\) gives
\[ \lambda g^{\flat}+(1-\lambda)g^{\sharp}\geqslant 0\quad \text{and}\quad \lambda g^{\flat}(j')+(1-\lambda)g^{\sharp}(j')\geqslant 0. \]
Multiplying the first above inequality by \(-g^{\flat}(j')\geqslant 0\) and the second by \(g^{\flat}\geqslant 0\), summing, and simplifying gives
\[ \frac{g^{\sharp}(j')}{g^{\flat}(j')} \leqslant \frac{g^{\sharp}}{g^{\flat}}. \]
Hence the result.
\end{enumerate}
\end{proof}
\begin{remark}
If our goal were solely to prove Theorem~\ref{theo:main}, we would not require an explicit expression for $a[t,\rho_{\alpha,\mu},R]$. Indeed, later on, we will only rely on the fact that $a[t,\rho_{\alpha,\mu},R]$ lies within $[0,\mu]$ and satisfies \eqref{item:oracle}. Therefore, the mere existence of such an $a[t,\rho_{\alpha,\mu},R]$ suffices.
This existence is actually quite straightforward using Blackwell’s dual condition (see Proposition~\ref{prop:dual}), which would have spared us the tedious case distinctions made in the proof of Proposition~\ref{prop:oracle}. However, we chose to give a formula for $a[t,\rho_{\alpha,\mu},R]$ because it allows us to derive an explicit expression for the $O(\varepsilon)$-uniform optimal strategy that we construct subsequently.\end{remark}


\section{The Resulting Strategy for Player I in the Absorbing Game}
\label{sec:result-strat-play}
We now define a strategy $\sigma=(\sigma_t)_{t\geqslant 1}$ for Player I in the absorbing game. 
For \((i_1,j_1,i_2,j_2,\dots )\in (I\times J)^{\mathbb{N}^*}\) and all \(t\geqslant 1\), recursively define
\begin{align*}
\tilde{R}_{t-1}^{(1)}&=\max_{}\left( 0,\ - \sum_{s=1}^{t-1}\alpha_sa_sg^{\sharp}(j_s) \right)\\
  \tilde{R}_{t-1}^{(2)}&=\max_{}\left( 0,\ \sum_{s=1}^{t-1}\alpha_s(\mu_s^{-1}a_s-1)g^{\flat}(j_s) \right) \\
  \alpha_t&=\prod_{s=1}^{t-1}\left( a_sq(\omega|x_{\lambda},j_s)+(1-a_s)q(\omega|x_0,j_s) \right) \\
  a_t&=\begin{cases}
0 & \text{if $\tilde{R}_{t-1}^{(1)}=\tilde{R}_{t-1}^{(2)}=0$}\\
\displaystyle \sup_{\substack{j\in J\\g^{\flat}(j)<0}} \frac{\mu^2_t g^\flat(j) \tilde{R}^{(2)}_{t-1}}{\mu_t  g^\flat(j) \tilde{R}^{(2)}_{t-1} - g^\sharp(j)\tilde{R}^{(1)}_{t-1} }&\text{otherwise},
\end{cases}
\end{align*}
where an easy induction proves that \(a_t\) is indeed a function of \((i_1,j_1,\dots,i_{t-1},j_{t-1})\), so that we can define
\[ \sigma_t(i_1,j_1,\dots,i_{t-1},j_{t-1})=a_tx_{\lambda}+(1-a_t)x_0. \]


\begin{theorem} 
\label{thm:2}
If $\mu_t= \varepsilon t^{-3/4}$ for all \(t\geqslant 1\), the above strategy \(\sigma\)
guarantees that for all $\tau\in \mathcal{T}$ and $T \geq 256 \lambda^{-4} \varepsilon^{-12}$,
\[ \gamma_T(\sigma,\tau)\geqslant - 8\varepsilon. \]
\end{theorem}

\begin{proof}
Thanks to Lemma~\ref{lm:pure-markov}, it is sufficient to prove the result for pure Markov strategies of Player II. Let \(\tau\) be such a strategy, and $(\omega_1,i_1,j_1,\dots,\omega_t,i_t,j_t,\dots,)\sim \mathbb{P}_{\sigma,\tau}$. For all \(t\geqslant 1\), consider above notation \(\alpha_t\) and \(a_t\) and
\[ \rho_t=\rho_{\alpha_t,\mu_t}\qquad \text{and}\qquad r_t=\rho_t(a_t,j_t)= \alpha_t
        \begin{pmatrix}
		             - a_t g^\sharp(j_t) \\
	 (\mu^{-1}_t a_t - 1)\, g^\flat(j_t)
		\end{pmatrix}. \]

Let us first establish a bound on \(\sum_{t=1}^{+\infty}\left\| r_t \right\|_{(t)}^2\).
For all $t\geqslant 1$, using Property~\eqref{item:bounds-action} from
Proposition~\ref{prop:oracle} as well as 
  $\alpha_t \in [0,1]$, $a_t\in [0, \mu_t]$ and $g^\flat(j_t),\,g^\sharp(j_t) \in [-1, 1]$ by Assumption~\ref{assumption}, we have
\[ \|r_t\|^2_{(t)} 
	  \leqslant |\alpha_t|^2   \big( a_t^2 + \mu^2_t (\mu^{-1}_t a_t - 1)^2\big)
	 \leqslant (2x^2_t + \mu^2_t - 2\mu_t a_t) \leqslant 3 \mu^2_t  \leqslant 3 \varepsilon^2 t^{-3/2}. \]
Therefore 
\[
\sum_{t=1}^{+\infty} \|r_t\|^2_{(t)} \le 3\varepsilon^2  \sum_{t=1}^{+\infty} t^{-3/2} \le 9\varepsilon^2.
\]
    
Then, combining the above definition of \(\sigma\) with Proposition~\ref{prop:oracle}, it holds that
\[ a_t=a\left[ t,\rho_t,\sum_{s=1}^{t-1}r_s \right],\qquad t\geqslant 1,  \]
and Corollary~\ref{cor:1} gives for all \(t\geqslant 1\),
\begin{align}
\label{eq:4}
 - \sum_{s=1}^{t-1} \alpha_s a_s\, g^\sharp(j_s) &= \sum_{s=1}^{t-1}r_s^{(1)}\leqslant \sqrt{\sum_{s=1}^{t-1}\left\| r_s \right\|_{(s)}^2}\leqslant 3\varepsilon\\
\label{eq:7}
  \frac{1}{T} \sum_{t=1}^T  \alpha_t (\mu_t^{-1} a_t-1)\, g^\flat(j_t) &= \frac{1}{T}\sum_{t=1}^Tr_t^{(2)}\leqslant \frac{1}{\mu_T T}\sqrt{\sum_{t=1}^T\left\| r_t \right\|_{(t)}^2}\leqslant \frac{3\varepsilon}{\mu_TT}\leqslant 3\varepsilon,
\end{align}
where the last inequality holds as soon as \(T\geqslant \varepsilon^{-4}\).
Combining \eqref{eq:4} (where the sum from \(s=1\) to \(t-1\) is needed for all \(t\geqslant 1\)) and \eqref{eq:7} (where only the sum from \(1\) to \(T\) is needed) with Proposition~\ref{prop:payoff_bound2} gives the result.
\end{proof}
\section{Comparison with Other Proofs of Theorem~\ref{theo:main}} \label{sec:comp}
Let us compare our proof of Theorem~\ref{theo:main} to three other proofs from the literature. 

\subsubsection*{The Original Proof of \citet{kohlberg74} for Finite Action Sets} 
\citet{kohlberg74} proved Theorem~\ref{theo:main} in the case where the action sets are finite. The proof uses a matrix theory approach~\citep{mills56} to reduce the problem to the case where Player 1 has only two actions, and then defines a strategy of Player I by making explicit the probability of playing each of the two actions as a function of the past history. Such a strategy is not a balanced strategy, and is not built from discounted optimal strategies. Hence, both the strategy and its analysis differ from our work. 

\subsubsection*{The Proof of \citet{MNR09} for Compact Action Sets}
The strategy built in \citet{MNR09} is based on discounted optimal strategies. Indeed, at each stage, the strategy plays optimally in a discounted game, where the discount factor depends on the past history. 
Unlike the present work, the discount factor can take infinitely many possible values, whereas in our case, it only takes two values (0 and $\lambda$, if we interpret the limit strategy $x_0$ as an optimal strategy for the discount factor 0). Moreover, the proof in \citet{MNR09} relies on the sophisticated machinery of \citet{MN81}, who proved the existence of the uniform value in general stochastic games. It also builds on the characterization of the limit value obtained by \citet{RS01}. In contrast, our proof is more self-contained, since it neither relies on the Mertens and Neyman strategy nor requires the existence of the limit value. 

\subsubsection*{The Proof of \citet{HIR21}}
\citet{HIR21} build an $\varepsilon$-uniform optimal strategy \textit{with finite memory and a clock}. At each stage, the mixed action played by such a strategy depends only on the state variable of a finite automaton. The state variable is updated from one stage to the next as a function of the previous action of Player 2, and its transitions depend on the stage. The automaton has three states. In two of them, the mixed action is a limit of discounted optimal strategies (\textit{careful action}), which is similar to our mixed action $x_0$. In the third state, the mixed action is a discounted optimal strategy for some small discount factor (\textit{risky action}), in the same fashion as our mixed action $x_\lambda$. In particular, the strategy built in \citet{HIR21} is a balanced strategy. One main difference with our strategy, is that in \citet{HIR21}, the weights that are put on the risky actions and the safe actions are induced by the probability distribution over the states of the automaton. Such a probability depends on the past history and on the transitions of the automaton, and the definition of such transitions is rather sophisticated. In contrast, the weights that our strategy puts on the risky action and on the safe action are derived from an approachability problem. This makes the analysis shorter, and it provides intuition on how the weights should be defined so that the strategy works. Moreover, the approachability framework that we develop is quite general and offers promising perspectives for possible generalizations to stochastic games. Another difference is that the strategy in \citet{HIR21} is built in blocks of increasing time lengths, whereas our strategy does not require such a progressive construction.

\section{Conclusion and Perspectives}
\label{sec:conclusion}

We introduced an extension of Blackwell's approachability framework
where the outcome function and the inner product vary with time, and studied
the corresponding Blackwell's algorithm.
In the case where the target set is an orthant, we presented a choice
of time-dependent inner products which yield different convergence speeds
for each coordinate of the average outcome vector.

We applied the latter case to the construction of 
$\varepsilon$-uniformly optimal strategies in absorbing games, thereby proposing a novel
application of online learning tools for solving games. We hope that the present work
can be extended into a new systematic approach for constructing
optimal strategies in a wider class of stochastic games.

We believe our framework will also find various applications in 
online learning and sequential decision problems as well.
For instance, an interesting direction would be the definition of a hybrid between the following two regret minimization algorithms, which enjoy different adaptive properties. Regret Matching~\citep{hart2001general} (which is based on Blackwell's approachability) and its variants have led to great success in the context of solving games, and 
the adaptive diagonal scalings of AdaGrad-Diagonal~\citep{mcmahan2010adaptive,duchi2011adaptive} (and its popular variants such as RMSprop and Adam) have demonstrated excellent performance in deep learning optimization and continuous (stochastic) optimization in general. As the adaptive scalings of AdaGrad-Diagonal can be written as time-dependent metrics, our framework should allow the definition of an algorithm that combines Regret Matching and (the dual averaging version of) AdaGrad-Diagonal, and hopefully inherit adaptive properties and excellent practical performance from both.

Further possible extensions of our framework include the adaptation to potential-based \citep{hart2001general} and regret-based~\citep{abernethy2011blackwell,shimkin2016online,kwon2021refined} algorithms.


\section*{Acknowledgments}
The authors are grateful to Sylvain Sorin and Vianney Perchet for valuable discussions
that helped improve this paper. This work was supported by the French Agence Nationale de la Recherche (ANR) under reference \texttt{ANR-21-CE40-0020} (CONVERGENCE project), and by a public
grant as part of the \emph{Investissement d'avenir project}, reference
\texttt{ANR-11-LABX-0056-LMH}, LabEx LMH. 

\bibliographystyle{abbrvnat}
\bibliography{bib}

\appendix

\section{On General Closed Convex Targets}
\label{sec:extens-gener-clos}

We here present one possible approach for reducing general closed convex targets to the conic case, which is inspired from~\cite{abernethy2011blackwell}. Let us consider the assumptions from Proposition~\ref{prop:dual}, but with \(\mathcal{C}\subset \mathbb{R}^d\) being a general nonempty closed convex set and not necessarily a cone. Further assume that \(\mathcal{C}\) satisfies the following Blackwell's condition: for all \(t\geqslant 1\), \(\rho\in \mathcal{R}\) and \(r\in \mathbb{R}^d\), there exists \(a\left[ t,g,r \right]\in \mathcal{A} \) such that for all \(b\in \mathcal{B}\),
\[ \left< \rho\left(a\left[ t,g,r \right],b  \right)-\pi_{(t)}^{\mathcal{C}}(r)  , r-\pi_{(t)}^{\mathcal{C}}(r) \right>_{(t)}\leqslant 0. \]
Then, by easily adapting from e.g.\ \cite[Theorem 1.3]{perchet2014approachability}, one can see that the above is equivalent to the same condition as in Proposition~\ref{prop:dual}:
\[ \forall \rho\in \mathcal{R},\ \forall b\in \mathcal{B},\ \exists a\in \mathcal{A},\ \rho(a,b)\in \mathcal{C}. \]

Now consider an auxiliary approachability problem with outcome functions \(\tilde{\rho}_t:\mathcal{A}\times \mathcal{B}\to \mathbb{R}^{d+1}\), defined as
\[ \tilde{\rho}_t(a,b)=(\rho_t(a,b),1),\qquad a\in \mathcal{A},\ b\in \mathcal{B},\ t\geqslant 1, \]
and target \(\tilde{\mathcal{C}}=\mathbb{R}_+(\mathcal{C}\times \left\{ 1 \right\} )\), which is a closed convex cone. Then, for \((a,b)\in \mathcal{A}\times \mathcal{B}\) and \(t\geqslant 1\), \(\rho_t(a,b)\in \mathcal{C}\) if and only if, \(\tilde{\rho}_t(a,b)\in \tilde{\mathcal{C}}\). This shows that the target of this auxiliary approachability problem is approachable because it satisfies the dual condition. The results from Sections~\ref{sec:gener-appr-probl} and \ref{sec:analysis} thus apply. Besides, \citet[Lemma 14]{abernethy2011blackwell} assures that the distance of the average outcome to \(\mathcal{C}\) in the original problem differ from the distance to the average outcome to \(\tilde{C}\) in the auxiliary problem by a factor 2 at most. Therefore, the convergence bounds from the auxiliary problem are easily transposed to the original problem.

\section{Postponed Proofs}
\label{sec:postponed-proofs}

\subsection{Proof of Lemma~\ref{lm:pure-markov}}
\label{sec:pure-markov}
To prove the lemma, it is enough to build, for each $T \geq 1$, a pure Markov strategy $\tau^*$ such that 
$\gamma_T(\sigma,\tau^*)=\min_{\tau \in \mathcal{T}} \gamma_T(\sigma,\tau)$. 
Let us take the point of view of Player II that aims at minimizing the payoff in the $T$-stage game, against the strategy $\sigma$. This problem can be viewed as a Markov Decision Process (1-Player stochastic game) \citep{Bellman_57}, where the decision-maker is Player II, and the state space is the product $\Omega \times \cup_{t \geq 1} J^{t-1}$. The first component of the state represents the state of the absorbing game (that we call \textit{original state}), while the second component encodes the sequence of actions played by Player II at some stage. Such a product variable is enough to describe the problem faced by Player II, due to the fact that the mixed action played by Player I at some stage depends only on past actions of Player II. This MDP admits a pure optimal Markov strategy \citep{DY79}, that is, a strategy that at each stage $t$, picks a mixed action that only depends on the original state $\omega_t$ and the sequence of past actions $(j_1,j_2,\dots,j_{t-1})$. The actions of Player II only have an influence when the original state is the non-absorbing state $\omega$, hence such a strategy is equivalent to a sequence of actions of Player II, that is, a pure Markov strategy in the absorbing game. This proves the lemma.

\subsection{Proof of Lemma~\ref{lm:alpha}}
\label{sec:alpha}

Let \(t\geqslant 2\).
If \(\mathbb{P}\left[ \omega_{t-1}=\omega \right]=0 \), then \(\mathbb{P}\left[ \omega_t=\omega \right]=0 \) by definition of the model. Otherwise we can write,
\begin{align*}
\mathbb{P}\left[ \omega_t=\omega \right] &=\mathbb{P}\left[ \omega_t=\omega\text{ and }\omega_{t-1} =\omega\right]\\
  &=\mathbb{P}\left[ \omega_t=\omega\,\middle|\,\omega_{t-1}=\omega \right]\times \mathbb{P}\left[ \omega_{t-1}=\omega \right]\\
  &=\left(\int_I^{}q(\omega|i,j_{t-1})\,\mathrm{d}(a_{t-1}x_{\lambda}+(1-a_{t-1})x_0)(i)  \right)\times \mathbb{P}\left[ \omega_{t-1}=\omega \right] \\
  &=q( \omega|a_{t-1}x_{\lambda}+(1-a_{t-1})x_0,j_{t-1} )\times \mathbb{P}\left[ \omega_{t-1}=\omega \right]\\
  &=\left(a_{t-1}\, q(\omega|x_{\lambda},j_{t-1})+(1-a_{t-1})q(\omega|x_0,j_{t-1})  \right)\times \mathbb{P}\left[ \omega_{t-1}=\omega \right],
\end{align*}
by linearity. In any case, the identity
\[ \mathbb{P}\left[ \omega_t=\omega \right]=\left(a_{t-1}\, q(\omega|x_{\lambda},j_{t-1})+(1-a_{t-1})q(\omega|x_0,j_{t-1})  \right)\times \mathbb{P}\left[ \omega_{t-1}=\omega \right] \]
holds, and the result follows from a simple induction.

\end{document}